%% file: skolem.tex
\documentclass{article}

\input{macros}

\title{Definable sets in Skolem arithmetic}

\author{Łukasz Kamiński}

\begin{document}
\maketitle

\begin{abstract}
    In this note, we present a characterization of sets definable in Skolem arithmetic,
    i.e., the first-order theory of natural numbers with multiplication.
    This characterization allows us to prove the decidability of the theory.
    The idea is similar to that of Mostowski; however, our characterization is new,
    and the proof relies on different combinatorial tools.
    The main goal of this note is to provide a simpler decidability proof than those previously known.
\end{abstract}

\subsection*{Introduction}

Skolem arithmetic (denoted by $Sk$) is the 
first order theory of the structure of
natural numbers with multiplication, 
i.e., $\langle \N; \cdot, = \rangle$.
The decidability of this theory
was first claimed by Skolem \cite{skolem},
but the first complete proof 
was published by Mostowski \cite{mos52}.
A few years later Feferman and Vaught
published a generalization of 
Mostowski's result \cite{fv59}.
Subsequently, the decidability
of Skolem arithmetic was also 
proven using quantifier elimination
in some extended language \cite{ceg81}.
There are also some other proofs
\cite{rackoff_skolem,nad81,hodgson}.

Mostowski’s method proceeds by characterizing sets definable 
in strong and weak powers of first-order theories; 
in particular, 
Skolem arithmetic can be regarded as a weak power of Presburger arithmetic \cite{presburger}.
In that framework the principal technical challenge is the elimination of quantifiers.

In this note we restrict attention to Skolem arithmetic itself.
However, the presented characterization of definable sets 
could be also extended to general weak powers.
While the high-level idea is reminiscent of Mostowski’s, 
the characterization and several proof steps differ: 
in our setting quantifier elimination is routine, and the main difficulty is handling
negations (complements).

\subsection*{Preliminaries}

\paragraph*{Vectors}

Let $\P$ be the set of prime numbers.
Fix $d \in \N$.
We use the following notational convention: 
for vectors $\vr u, \vr v \in \N^d$
by $\vr u \vr v$ we denote their pointwise product,
i.e., 
\begin{equation*}
    \begin{pmatrix}
    u_1 \\ \vdots \\ u_d
    \end{pmatrix}
    \begin{pmatrix}
        v_1 \\ \vdots \\ v_d
    \end{pmatrix}
    = 
    \begin{pmatrix}
        u_1v_1 \\ \vdots \\ u_d v_d
    \end{pmatrix}.
\end{equation*}
Furthermore, for $p \in \P$
and $\vr v = (v_1, \ldots, v_d)^T \in \N^d$
let 
$
    p^{\vr v} = 
    \begin{pmatrix}
        p^{v_1},
        \ldots, 
        p^{v_d}
    \end{pmatrix}^T
$.
For $N \in \N$ and $p \in \P$
we denote by $\v{p}{N}$ the greatest $n \in \N$
such that $p^n | N$.
For a vector $\vr v = (v_1, \ldots, v_d)^T \in \N^d$
we denote $\v{p}{\vr v} = (\v{p}{v_1}, \ldots, \v{p}{v_d})^T$.
Observe that $\prod_{p \in \P} p^{\v{p}{\vr v}} = \vr v$.

Let $[d] = \set{1, \ldots, d}$.
For $\vr v \in \N^d$ and a coordinate $x \in [d]$
we denote by $\proj{x}{\vr v}$ the vector obtain from $\vr v$
by projecting away the coordinate $x$.
Moreover, for $S \subseteq \N^d$
let $\proj{x}{S} = \setof{\proj{x}{\vr v}}{\vr v \in S} \subseteq \N^{d-1}$.
We use a convention that $\N^0 = \set{\varepsilon}$,
where $\varepsilon$ is the unique empty tuple,
and we denote $\top = \set{\varepsilon}$.
For $S \subseteq \N$ and a coordinate $x=1$
let $\proj{x}{S} = \top$ if $S \ne \emptyset$
and $\proj{x}{S} = \emptyset$ otherwise.

Let $i,j \in [d]$.
We denote by $\vr e_i, \vr e_{ij} \in \N^d$
the vectors defined as follows:
\begin{equation*}
    \vr e_{ij}(k) =
    \begin{cases}
        1, \text{ if } k = i \text{ or } k = j \\
        0, \text{ if } k \ne i,j
    \end{cases}
    \qquad
    \vr e_{i}(k) =
    \begin{cases}
        1, \text{ if } k = i \\
        0, \text{ if } k \ne i
    \end{cases}
    .
\end{equation*}
We assume a convection that $\vr e_{ij} = \vr e_i$ if $i = j$.

\paragraph*{Graph matchings}

In this paragraph we present a graph theoretic lemma that we will use.
Recall that $G = (L \cup R, E)$ is a bipartite graph if
$E \subseteq L \times R$.
A \emph{matching} 
of a set $S \subseteq L \cup R$ of nodes is a set 
$M \subseteq E$ of pairwise non-adjacent edges that covers all nodes in $S$.

\begin{lemma}[{\cite[Lem.~4]{HLT17}}]\label{lem:matching}
    Let $G = (L \cup R, E)$ be a bipartite graph. 
    If there is a matching of $L' \subseteq L$ and a matching of $R' \subseteq R$
    then there is a matching of $L' \cup R'$.
\end{lemma}

\subsection*{Definable sets}

We assume that the reader is familiar
with Presburger arithmetic and semilinear sets
\cite{presburger,semilin}.
We say that a set $S \subseteq \N^d$
is \emph{\lin} if there exist
semilinear sets 
$\alpha_1, \ldots, \alpha_n, \alpha \subseteq \N^d$
with $\vr 0 \in \alpha$ and $\vr 0 \not\in \alpha_i$ for $i \in [n]$
such that
\begin{equation*}
    S = 
    \setof{
        p_1^{\vr v_1} \ldots p_n^{\vr v_n} q_1^{\vr u_1} \ldots q_k^{\vr u_k}
    }{
        k \in \N, \
        \vr v_i \in \alpha_i, \
        \vr u_j \in \alpha \text{ for } i \in [n], j \in [k]
    },
\end{equation*}
where $p_1, \ldots, p_n, q_1, \ldots, q_n$
range over pairwise distinct prime numbers.
In other words, $\vr w \in S$ if and only if
there exist pairwise distinct primes $p_1, \ldots, p_n$ such that $\v{p_i}{\vr w} \in \alpha_i $ for $i \in [n]$
and $\v{q}{\vr w} \in \alpha$ for all other $q \in \P$ with $\v{q}{\vr w} \neq \vr 0$.
We say that formulas $\alpha_1, \ldots, \alpha_n, \alpha$
\emph{define} the \lin{} set $S$,
and we write $S = \Def(\alpha_1, \ldots, \alpha_n; \alpha)$.
In particular, if $n=0$ then we write $S = \Def(\alpha)$.

A set $S$ is \emph{\slin} if there are
\lin{} sets $S_1, \ldots, S_m$ 
such that $S = \bigcup_{i=1}^m S_m$.

\smallskip

Now we are ready to formulate the main theorem of this note.
\begin{thm}\label{thm:main}
    A subset of $\N^d$ is definable in Skolem arithmetic
    exactly when it is \slin. 
    Furthermore, given a formula $\Phi$ of Skolem arithmetic,
    the \slin{} representation of the set defined by $\Phi$
    is computable.
\end{thm}

We will omit the proof of the easier implication: 
every \slin{} set is definable in Skolem arithmetic.
Now we focus on the reverse implication.
The proof is by induction.
If suffices to show that atomic formulas define 
\slin{} sets,
 and if sets defined by formulas
$\Phi$, $\Psi$ are \slin{} then also 
sets defined by $\Phi \wedge \Psi$, $\neg \Phi$ and $\exists_X \Phi$ are so,
assuming that $X$ is a free variable of $\Phi$.
The three properties are equivalent 
to the following closure properties of \slin{} sets:
closure under intersection, complement and 
projecting away a coordinate.

Note that we could prove that \slin{} sets 
are closed under union instead of intersection.
The closure under union follows directly from the definition
of \slin{} sets. 
However, we prove the closure under intersection anyway,
as the proof of closure under complement
will use this fact.

\subsection*{Closure properties}

This section is devoted to the proofs of closure 
properties of \slin{} sets.
We start with proving that atomic formulas
of Skolem arithmetic define \slin{} sets.

\begin{itemize}
    \item {\bf Equality}

    Suppose that $S \subseteq \N^d$
    is defined by a formula $x_i=x_j$,
    where $x_i, x_j$ are variables corresponding to coordinates $i,j \in [d]$, respectively.
    Let $\alpha$ be a semilinear set defined by
    a Presburger formula $x_i = x_j$.
    Roughly speaking, two nonnegative integers 
    are equal if and only if each prime
    appears with the same exponent in the factorization
    of both. Hence,
    \begin{equation*}
        S = \Def(\alpha).
    \end{equation*}

    \item {\bf Multiplication}
    
    This case is similar to the previous one.
    Suppose that $S \subseteq \N^d$
    is defined by a formula $x_i \cdot x_j = x_k$,
    where $x_i, x_j, x_k$ are variables corresponding to coordinates $i,j,k \in [d]$, respectively.
    Note that $i,j,k$ are not necessarily pairwise different.
    Let $\alpha$ be a semilinear set defined by a Presburger formula
    $x_i + x_j = x_k$.
    We observe that
    \begin{equation*}
        S = \Def(\alpha).
    \end{equation*}

    \item {\bf Intersection}
    
    Observe that 
    $
    \left(\bigcup_{i \in I} S_i\right) \cap \left(\bigcup_{j \in J} S_j'\right) 
    = \bigcup_{i \in I, j \in J} S_i \cap S_j'
    $.
    Hence, it suffices to prove that intersection of two \lin{} sets
    is \slin{}.
    
    Let $S = \Def(\alpha_1, \ldots, \alpha_{n}; \alpha)$
    and
    $S' = \Def(\beta_1, \ldots, \beta_{m}; \beta)$
    for semilinear sets $\alpha_1, \ldots, \alpha_{n}, \alpha$
    and $\beta_1, \ldots, \beta_{m}, \beta$.
    We call a subset $P$ of $(\set{0} \cup [n]) \times (\set{0}\cup [m])$ 
    \emph{correct}
    if the following condition is satisfied:
    for every $i \in [n]$ there exists exactly one $j \in \set{0} \cup [m]$
    such that $(i, j) \in P$
    and 
    for every $j \in [m]$ there exists exactly one $i \in \set{0} \cup [n]$
    such that $(i, j) \in P$.
    Denote by $\mathcal{P}$ the family of all correct sets.
    Then the set $S \cap S'$ is represented as 
    \slin{} set as follows
    \begin{equation*}
        S \cap S' = \bigcup_{P \in \mathcal{P}} 
        \Def\left(\left(\alpha_i \cap \beta_j\right)_{(i,j) \in P}; \alpha \cap \beta \right).
    \end{equation*}
    Indeed, if $P$ is correct then
    any element of
    $\Def\left(\left(\alpha_i \cap \beta_j\right)_{(i,j) \in P}; \alpha \cap \beta \right)$
    belongs to both $S$ and $S'$.
    On the other hand, for every $\vr w \in S \cap S'$
    there exists $P \in \mathcal{P}$ which witness 
    that $\vr w$ belongs to the above \slin{} set.

    \item {\bf Complement}
    
    For a set $A \subseteq \N^d$ let $A^c = \N^d \setminus A$.
    Observe that $(\bigcup_{i = 1}^{m} S_i)^c = \bigcap_{i=1}^m S_i^c$.
    As we already proved that \slin{} sets are closed under intersection,
    it suffices to show that a complement of a \lin{} set is \slin{}.

    Let $\vr w \in \N^d$ and $S = \Def(\alpha_1, \ldots, \alpha_n; \alpha)$.
    Let $P$ be the set of primes $p$ satisfying
    $\v{p}{\vr w} \neq \vr 0$ and $P_0 \subseteq P$
    be its subset that consists of primes $p$ such that
    $\v{p}{\vr w} \not\in \alpha$.
    Consider a bipartite graph $G = (P \cup [n], E)$, 
    where $(p, i) \in E$ if and only if
    $\v{p}{\vr w} \in \alpha_i$.
    We will use the following fact.
       
    \begin{cl}\label{cl:suff}
        $\vr w \in S$ if and only if the following conditions are satisfied:
%
    \begin{enumerate}
        \item[(a)] 
        there is a matching of $P_0$ in $G$;
        \item[(b)]
        there is a matching of $[n]$ in $G$.

    \end{enumerate}
    \end{cl}
    \begin{proof}
        Observe that if $\vr w \in S$ then conditions (a),(b) are satisfied.
        Let us focus on the opposite direction.

        From Lemma \ref{lem:matching} we conclude that
        there exists a matching $M$ of $P_0 \cup [n]$ in $G$.
        Let $p_i \in P$ be the prime such that $(p_i, i) \in M$ for $i \in [n]$.
        Observe that $\v{q}{\vr w} \in \alpha$
        for every $q \in P \setminus P_0$.
        Therefore,
        \begin{equation*}
            \vr w = 
            \prod_{i \in [n]} p_i^{\v{p_i}{\vr w}} 
            \prod_{q \in P \setminus \set{p_1, \ldots, p_n}} q^{\v{q}{\vr w}}
            \in S.
        \end{equation*}
        Indeed, $\v{p_i}{\vr w} \in \alpha_{i}$ for $i \in [n]$
        and $\v{q}{\vr w} \in \alpha$ for $q \in P \setminus \set{p_1, \ldots, p_n}$.
    \end{proof}

    We continue the proof that a complement of a \lin{} set is \slin{}.
    Let $S = \Def(\alpha_1, \ldots, \alpha_n; \alpha)$.
    Then $\vr w \not\in S$ if and only if one of conditions (a),(b)
    from Claim \ref{cl:suff} is not satisfied.
    Let $S_a, S_b \subseteq \N^d$ be the sets of vectors 
    that does not satisfy conditions (a), (b) respectively.
    As $S^c = S_a \cup S_b$, it suffices to show that sets
    $S_a, S_b$ are \slin{}.
    \begin{itemize}
        \item (a) fails
        
        There may be two reasons why (a) does not hold.
        The first one is that $|P_0| > n$.
        We define the set of vectors $\vr w$
        such that $\v{p}{\vr w} \not\in \alpha$
        for at least $n+1$ primes $p$:
        \begin{equation*}
            \Def(\underbrace{\alpha^c, \ldots, \alpha^c}_{n+1}; \N^d).
        \end{equation*} 
        In the second case we assume that there is no
        matching of $P_0$ in $G$ but $|P_0| \leq n$.
        From the Hall's theorem
        we conclude that there exists $P_0' \subseteq P_0$
        with $|P_0'| = n' \leq n$ such that for less than
        $n'$ indexes $i \in [n]$ the condition
        $p \in \alpha_i$ is satisfied for some $p \in P_0'$.
        Equivalently,
        there is $n-n'+1$ sets out of
        $\set{\alpha_1, \ldots, \alpha_n}$
        and $n'$ primes $p \in P$
        such that $\v{p}{\vr w}$ neither belongs to any of these sets nor to $\alpha$.
        For $I \subseteq [n]$ let
        $\beta_I := \alpha^c \cap \bigcap_{i \in I} \alpha_i^c$.
        Then we define the set of vectors $\vr w$ satisfying this condition
        as follows:
        \begin{equation*}
            \bigcup_{n' \leq n} \
            \bigcup_{\substack{I \subseteq [n] \\ |I| = n-n'+1}}
            \Def(\underbrace{\beta_I, \ldots, \beta_I}_{n'}; \N^d).
        \end{equation*}

        \item (b) fails
        
        Again we use the Hall's theorem to conclude that
        there exists $I \subseteq [n]$ with the property that there is
        less than $|I|$ primes $p \in P$
        such that $\v{p}{\vr w}$ belongs to some set from the family $\setof{\alpha_i}{i \in I}$.
        Let $\gamma_I = \bigcup_{i \in I} \alpha_i$. Then
        \begin{equation*}
            S_b = \bigcup_{I \subseteq [n]} \ \bigcup_{n' < |I|}
            \Def(\underbrace{\gamma_I, \ldots, \gamma_I}_{n'}; \gamma_I^c).
        \end{equation*}
    \end{itemize}

    \item {\bf Projecting away a coordinate}
    
    It suffices to prove that a set obtained by
    projecting away a coordinate from a \lin{} set
    is \lin{}.
    Suppose that
    $S = \Def(\alpha_1, \ldots, \alpha_n ; \alpha)$
    and $x \in [d]$ is a coordinate.
    Then
    \begin{equation*}
        \proj{x}{S} = \Def(\proj{x}{\alpha_1}, \ldots, \proj{x}{\alpha_n} ; \proj{x}{\alpha}).
    \end{equation*}
\end{itemize}

\subsection*{Conclusion}

Theorem \ref{thm:main} implies
the following decidability result.

\begin{thm}\label{thm:dec}
    Skolem arithmetic is decidable.
\end{thm}
\begin{proof}
    Let $\psi (x_1, \ldots, x_d)$ be a formula in the language 
    of Skolem arithmetic.
    The goal is to decide if $Sk \models \psi(a_1, \ldots, a_d)$
    for a given $a_1, \ldots, a_d \in \N$.
    Due to theorem \ref{thm:main}
    there is an effective procedure
    to compute the set defined by $\psi$,
    $S = \Def(\alpha_1, \ldots, \alpha_n ; \alpha)$
    for some semilinear sets $\alpha_1, \ldots, \alpha_n, \alpha$.
    Hence, it suffices to check if $\vr a = (a_1, \ldots, a_d)^T \in S$.
    The set of primes $p$ for which $\v{p}{\vr a} \ne \vr 0$ is finite,
    so it is enough to check if conditions from Claim \ref{cl:suff} are satisfied.

    Note that this procedure works also if $\psi$ is a sentence,
    i.e., $d=0$. 
    In that case we need to check if $S = \top$,
    which holds if and only if 
    $\alpha_1 = \ldots = \alpha_n = \top$.
\end{proof}

\bibliography{bib}

\end{document}

%% file: macros.tex
\usepackage{amsmath}
\usepackage{amssymb}
\usepackage[english]{babel}
\usepackage{amsthm}
\usepackage{hyperref}

\newcommand{\N}{\mathbb{N}}
\renewcommand{\P}{\mathbb{P}}

\newcommand{\set}[1]{\{#1\}}
\newcommand{\setof}[2]{\left\{#1 \ | \ #2 \right\}}
\renewcommand{\v}[2]{V_{#1}(#2)}
\newcommand{\proj}[2]{\pi_{#1}(#2)}
\newcommand{\vr}[1]{\mathbf{#1}}

\newcommand{\Def}{\textnormal{Def}}
\newcommand{\lin}{skolemian}
\newcommand{\slin}{semiskolemian}

\newtheorem{thm}{Theorem}
\newtheorem{cl}[thm]{Claim}
\newtheorem{lemma}[thm]{Lemma}